\documentclass[12pt]{amsart}
\usepackage{graphicx}
\usepackage{xy} \xyoption{all}
\usepackage[utf8]{inputenc}
\usepackage{pgf,tikz,pgfplots}
\pgfplotsset{compat=1.14}
\usepackage{mathrsfs}
\usetikzlibrary{arrows}
\usepackage{tikz-network}

\newtheorem{lemma}{Lemma}[section]
\newtheorem{corollary}[lemma]{Corollary}
\newtheorem{theorem}[lemma]{Theorem}
\newtheorem{proposition}[lemma]{Proposition}

\theoremstyle{definition}
\newtheorem{definition}{Definition}[section]

\newtheorem{remarks}[definition]{Remarks}
\newtheorem{remark}[definition]{Remark}

\newcommand{\Q}{\mathbb{Q}}
\newcommand{\Z}{\mathbb{Z}}
\newcommand{\N}{\mathbb{N}}
\newcommand{\F}{\mathbb{F}}
\def\su{\mathop{\sum}\limits_{i=1}^{n}}

\def\sm{\mathop{\sum}\limits_{i=1}^{m}}
\def\smj{\mathop{\sum}\limits_{j=1}^{t}}
    
\def\suo{\mathop{\sum}\limits_{j=0}^{n-1}}

\DeclareMathOperator{\ann}{ann}

\DeclareMathOperator{\Hom}{Hom}     
\def\a{\alpha}
\def\b{\beta}
\def\f{\phi}
\def\s{\sigma}

\def\p{\pi}
\def\l{\lambda}
\def\g{\gamma} 

\def\m{\mu}
\def\n{\nu}

\begin{document}
\title{Leavitt path algebras as flat bimorphic localizations}
\author{P. N. \'Anh}
\address{R\'enyi Institute of Mathematics, Hungarian Academy of
Sciences, 1364 Budapest, Pf. 127 Hungary} \email{anh@renyi.hu}
%\author{J. E. van den Berg}
%\address{Department of Mathematics and Applied Mathematics, University of Pretoria, Private Bag X20, Hatfield, Pretoria 0028, South Africa} \email{john.vandenberg@up.ac.za}
\author{M. F. Siddoway}
\address{Department of Mathematics and Computer Science, Colorado College,
Colorado Springs, CO 80903.} \email{msiddoway@coloradocollege.edu}
%\thanks{.}
\thanks{The first author was partially supported by National Research, Development and Innovation Office NKHIH K119934 and  K132951, by both Vietnam Institute for Advanced Study in Mathematics (VIASM) and  Vietnamese Institute of Mathematics. 
The second author is supported by a generous grant from the Colorado College Natural Sciences Division.}
\date{\today}

\begin{abstract}

Refining an idea of Rosenmann and Rosset we show that the now widely studied classical Leavitt algebra $L_K(1,n)$ over a field $K$ is a ring of right quotients of the unital free associative algebra of rank $n$ with respect to the perfect Gabriel topology defined by powers of an ideal of codimension 1, providing a conceptual, variable-free description of $L_K(1, n)$. This result puts Leavitt (path) algebras on the frontier of important research areas in localization theory, free ideal rings and their automorphism groups, quiver algebras and graph operator algebras. As applications one obtains a short, transparent proof for the module type $(1,n)\, (n\geq 2)$ of Leavitt algebra $L_K(1, n)\, (n\geq 2)$, and the fact that Leavitt path algebras of finite graphs are rings of quotients of corresponding ordinary quiver algebras with respect to the perfect Gabriel topology defined by powers of the ideal generated by all arrows and sinks. In particular, the Jacobson algebra of one-sided inverses, that is, the Toeplitz algebra, can also be realized as a flat ring of quotients, further illuminating the rich structure of these beautiful, useful algebras.
\end{abstract}
\maketitle
\section{Introduction}
\label{int}

W. G. Leavitt introduced the extraordinarily insightful notion of \emph{"module type"}, an important invariant of rings, in the late fifties of the last century. He showed that a unital ring is either a ring with IBN, that is, every free module has a unique rank, or is of module type $(m, n) (1\leq m < n)$ where $m, n$ are the smallest integers with respect to the property that the free modules generated by a basis having $m$ and $n$ elements, respectively, are isomorphic. He \cite{leav1} constructed (for the sake of simplicity) a universal algebra of module type $(1, n) (1<n)$ over the field $\F_2$ of 2 elements. There are several proofs for this result provided by Cohn \cite{c1}, Corner \cite{cor1}, \cite{cor2}  (together with another that is unpublished, as far as we know) and for $C^*$-algebras by Cuntz \cite{cu}. Analogous problems for Boolean algebras are discussed in Givent and Halmos' book \cite{gh1}, Chapters 27 and 45. Independently of Leavitt's work, Cuntz \cite{cu} invented the twin $C^*$-algebra ${\mathcal O}_n (n>1)$ which has dominated research in operator algebras in the last half century.  The intensive research in operator algebras initiated by Cuntz's fundamental result led first to operator graph algebras and then inspired ring theorists to consider "Leavitt path algebras of directed graphs" (for a good account see \cite{aas}) which has become an active subject in ring theory. The motivation for our work is the telling observation by Rosenmann and Rosset \cite{rr} that the module type of the fc-localization of the free unital associative algebra of rank 
$n\, (n\geq 2)$ is $(1,n)$. Namely, we show on one side that the canonical inclusion of a free associative algebra $A$ of rank $n(>0)$ over an arbitrary field $K$, that is, the quiver algebra $KE$ of a graph $E$ consisting of one vertex and $n$ loops, into the associated Leavitt path algebra $L_K(E)$, denoted as $L_K(1,n)$ is a flat bimorphism in the category of $K$-algebras. On the other side, we construct precisely classical Leavitt algebras $L_K(1,n)\, (n>0)$ as flat rings of right quotients of free associative algebras with respect to a perfect Gabriel (two-sided ideal) topology defined by powers {\bf $I^l \,(l\in \N)$ } of an ideal  $I \vartriangleleft A$ of codimension 1, providing a conceptual, variable-free description of $L_K(1, n)$. Moreover, we show that the two topologies, i.e., the ideal topology defined by the powers of an ideal  $I\vartriangleleft A$ of codimension 1 and the one defined in Theorem \ref{flatepi}(b) with respect to a corresponding flat bimorphism from $A$ into the associated quotient ring which is isomorphic to the Leavitt algebra $L_K(1,n) (n>1)$, coincide. This refined observation connects Leavitt (path) algebras to some central parts of other related research, like quiver algebras of (finite) directed graphs, localization theory, free ideal rings and epimorphisms of rings. It shows also the naturality of Cuntz-Krieger conditions (CK1) and (CK2) in the definition of both Leavitt and operator graph algebras. As further consequences, one obtains a new proof for some basic properties of Leavitt algebras and we show that the Leavitt path algebra of a finite directed graph is a perfect localization of the quiver algebra with respect to the Gabriel topology consisting of certain well-defined finitely generated essential right ideals. Consequently, a Leavitt path algebra of a finite directed graph is flatly bimorphic to the ordinary quiver algebra. It is also worth remarking that the determination of the module type of non-IBN, projective-free rings is frequently equivalent to the computation of an associated Grothendieck group, emphasizing the natural importance of $K$-theory in the study of Leavitt path algebras. Therefore, Leavitt's notion of "module type" $(m, n)$ seems to be a unifying idea that connects different algebraic realms like ring theory, $K$-theory, operator algebras and localization theory.

As a result, Leavitt path algebras can be viewed from two different vantage points. They are good examples for (until now nonstandard, but) natural localizations, and thus the techniques of localization, and results and methods from quiver algebras can be applied in the study of Leavitt path algebras. Cohn's
localization by inverting matrices makes it transparent that a canonical imbedding of the free associative algebra on a set $X=(x_{ij}) (i=1, \cdots, m; j=1, \cdots n,; 1<m<n)$ of free generators $x_{ij}$ into the universal Leavitt algebra of module type $(m,n)$, that is, the finitely presented algebra generated by generators $X=(x_{ij}), Y=(y_{ji})$ subject to $XY=1_m; YX=1_n$, is a bimorphism; that is, both a monomorphism and an epimorphism, in the category of algebras. However, we do not know whether this bimorphism is flat and so the relation to Gabriel's localization awaits further clarification. 
We also note that Bergman \cite{b1}, \cite{b2} extended Cohn's idea and provided yet another context for Leavitt's results in a more general setting. He also observed a connection between his universal constructions and localizations but did not realize the striking fact that Leavitt algebras are (flat bimorphic) localizations!

\section{Preminilaries: localization, digraphs and their algebras}
\label{premi}
\emph{A word about terminology}. All fields are commutative. All algebras, modules are unital and associative over a field unless stated otherwise. 
\emph{Ideals} and \emph{modules} are considered with respect to algebras, that is,
they are also at the same time vector spaces over a field. 
\emph{Finitely presented} modules are factors of finitely generated modules by finitely generated submodules.  
$A(1-x)$ denotes always the
left ideal $\{r-rx \,|\, r\in A\}$ for an element $x$ of a ring $A$ and similarly for the right ideal $(1-x)A$.  For further undefined notions for rings, localizations or for Leavitt path algebras we refer to monographs \cite{s1} and \cite{aas}, respectively. 

The theory of \emph{Rings of quotients} was introduced independently by Findlay and Lambek \cite{fl} and Utumi \cite{u1} in the late 1950's. Utumi's work is definitely important for use in Leavitt path algebras by permitting rings without identity. Namely, a ring $Q$ (not necessarily with identity) is a \emph{ring of (right) quotients} of a subring $R$, and in this case $R$ is called \emph{dense (on the right)} in $Q$ if for any two elements $q_1, q_2 \in Q, q_1\neq 0$, there is $r\in R$ such that $q_1r\neq 0, q_2r\in R$. It is trivial that the left annihilator of a dense subring is precisely the zero ideal, and dense right ideals are essential.
Alternative approaches to localization via torsion theory, calculus of fractions and quotient categories by localizing Serre subcategories were developed later in the sixties by Gabriel, Lambek, Morita, etc. For details we refer to books 
\cite{faith1}, \cite{pop}, \cite{s1} and papers by Lambek \cite{lab}, Morita \cite{mori1}, \cite{mori2}, \cite{mori3}, \cite{mori4}. These citations are primarily meant to direct the reader to essential developments of the theory and to the rich collection of references included by the authors.

Recall that a ring homomorphism $\f:A\rightarrow B$ is called an \emph{epimorphism} if for any ring $C$ and ring homomorphisms $\a, \b: B\rightarrow C, \a \f=\b \f$ implies $\a=\b$. Dually, a ring homomorphism $\f:A\rightarrow B$ is a \emph{monomorphism} if for any ring $C$ and any two ring homomorphisms $\a, \b: C\rightarrow A$, an equality $\f \a=\f \b$ implies $\a=\b$. Ring epimorphisms are not necessarily surjective but ring monomorphisms are always injective.  Namely, if there is $0\neq a\in A$ with $\f(a)=0$, then $\a, \b:\Z[X]\longrightarrow A$ defined by putting $\a(X)=0$ and $\b(X)=a$, respectively, are two different ring homomorphisms from $\Z[X]$ to $A$ satisfying $\f \a=\f \b$. Two rings are \emph{bimorphic} if there is a ring homomorphism between them which is both a monomorphism and an epimorphism. For example, the ring $\Z$ of integers and the field $\Q$ of rationals are bimorphic. An epimorphism $\f:A\rightarrow B$ is \emph{flat} if $_AB$ is a flat left $A$-module. In this case $B$ is called a \emph{perfect right localization} or a \emph{flat epimorphic right ring of quotients} of $A$. For the sake of completeness and because of its importance in our study, we quote here a characterization of flat epimorphisms which has been proved by a numbers of authors (for example, Findlay, Knight, Lazard, Morita, Popescu and Spircu, etc.), using various methods independently almost at the same time.
\begin{theorem}[Theorem XI.2.1 \cite{s1}]
\label{flatepi} Let $\f:A\rightarrow B$ be a ring homomorphism. The following assertions are equivalent:

(a) $\f$ is an epimorphism and makes $B$ into a flat left $A$-module.

(b) The family $\frak F$ of right ideals $\frak a$ of $A$ such that $\f(\frak a)B=B$ is a Gabriel topology, and there is a ring isomorphism $\s:B\rightarrow A_{\frak F}$ such that $\s \f:A\rightarrow A_{\frak F}$ is the canonical homomorphism

(c) The following two conditions are satisfied:

(i) For every $b\in B$ there exist $s_1, \cdots, s_n \in A$ and $b_1, \cdots, b_n \in B$ such that $b\f(s_i)\in \f(A)$ and $\su \f(s_i)b_i=1$.

(ii) If $\f(a)=0$, then there exist $s_1, \cdots, s_n \in A$ and $b_1, \cdots, b_n \in B$ such that $as_i=0$ and $\su \f(s_i)b_i=1$.
\end{theorem}

Note that assertion $(b)$ simply means that $B$ is a quotient ring of $A$ with respect to the Gabriel topology $\frak F$, and $\frak F$ is the finest among the Gabriel topologies $\frak T$ of $A$ such that the ring $A_{\frak T}$ of right quotients of $A$ with respect to $\frak T$ is isomorphic to $B$. Moreover, for our aim it is worth noting that the verification of the implication $(c)\Rightarrow (a)$ is elementary, and does not require any knowledge from localization theory. Condition $(i)$ of assertion $(c)$ has an important consequence in our study.
\begin{corollary}[Exercise 11 of IV.4.16 \cite{pop} p.269]
\label{rightid} If $\f: A\rightarrow B$ is a flat epimorphism, then every right ideal $\frak b$ of $B$ satisfies $\frak b=\f({\f}^{-1}(\frak b))B$.  Therefore the associated Gabriel topology admits a basis consisting of finitely generated ideals which are also dense whence essential provided that $\f$ is also one-to-one, i.e., $\f$ is a flat bimorphism. In particular, if $\f$ is one-to-one, then $\frak b=(\frak b \cap A)B$ holds by identifying $\f(a)$ with $a$ for every $a\in A$.
\end{corollary}
\begin{proof} The claim follows immediately from the observation that for any $b\in B$ and $s_i \in A, b_i \in B$ as in $(i)$ of assertion $(c)$, one has
$$b=b\su \f(s_i)b_i=\su (b\f(s_i))b_i.$$
The above equation shows also that $B$ is therefore obtained by a kind of generalized calculus of fractions. The other claims are obvious in view of the fact that a right ideal $\sum s_iA $ generated by the  $s_i$ is open by assertion $(b)$ of Theorem \ref{flatepi}, and the equality $\su \f(s_i)b_i=1$.
\end{proof}
If $\f: A\rightarrow Q$ is a flat bimorphism, then $Q$ becomes a subring of the maximal ring $Q_{\rm max}(A)$ of right quotients of $A$ by identifying each $a\in A$ with $\f(a)\in Q$. It is well-known that $Q_{\rm max}(A)$ has a largest subring  $Q_{\rm tot}(A)$ (called the \emph{maximal flat epimorphic ring of right quotients} of $A$) which is a flat epimorphism of $A$ and contains all flat bimorphisms of $A$. It is also well-known, and in fact, not hard to see that every element $q$ of $Q_{\rm max}(A)$ uniquely determines a largest right ideal
${\rm dom}(q)=\{a\in A\, |\, qa\in A\}$ of $A$, called a \emph{maximal right ideal of definition of} $q$, and $q$ can be identified with an $A$-homomorphism $q:{\rm dom}(q)\longrightarrow A: a\mapsto qa=q(a)$. This observation greatly simplifies notations when working inside maximal rings of quotients.
  
Since terminology in graph theory is notoriously nonstandard, we fix
necessary notations and concepts to help avoid confusion. 
We then remind the reader of the definitions of both quiver and
Leavitt path algebras, and make some germane observations about these algebras which will be useful in the sequel.

A \emph{directed graph} (or simply \emph{digraph}) is a quadruple $E=(E^0, E^1, s, r)$ of a non-empty \emph{vertex} set $E^0$, an arbitrary \emph{arrow} set $E^1$, and 
\emph{source} and \emph{range functions} $s, r:E^1\rightarrow E^0$. A vertex
$v\in E^0$ is \emph{singular} if it is either a \emph{sink} or an \emph{infinite emitter} 
 according to $|s^{-1}(v)|= 0 $ or $|s^{-1}(v)| =\infty$, respectively. 
$v$ is a \emph{regular vertex} if it is not singular.
A \emph{finite path} $\a$ of \emph{length} $n=|\a|>0$ from a {\it source} $s(\a)=v_0$ to a {\it range} $r(\a)=v_n$ is a sequence $\a=\overset{v_0}{\bullet} \overset{a_1}{\longrightarrow} \overset{v_1}{\bullet}\longrightarrow \cdots \overset{v_{n-1}}{\bullet} \overset{a_n}{\longrightarrow} \overset{v_n}{\bullet}$, written as a finite word $\a=a_1\cdots a_n$, of $n$ 
arrows $a_i\in E^1$ satisfying $s(a_{i+1})=v_i=r(a_i)\, (i=1,\cdots, n-1)$. Moreover,
$a_1\cdots a_i \, (0\leq i\leq n)$ is called a \emph{head} $h_{\a}(i) \, (h_{\a}(0)=v_0)$  of length $i$ while $a_{i+1}\cdots a_n$ is called a \emph{tail} $t_{\a}(i) \, (t_{\a}(n)=v_n)$ of colength $i$ of $\a$, respectively. Sometimes, for the sake of simplicity, we will omit the index $\a$ in functions $h, \, t$ when the meaning is clear. 
Every vertex $v\in E^0$ is, by convention, a path of length $0$ from $v$ to $v$. 
The set of all finite paths is denoted by $F(E)$.

A useful device in applications of digraphs is the composition of paths. Namely, if $\a=a_1\cdots a_n$ and $\b=b_1\cdots b_m$ are paths of length $n, m\geq 0$, respectively, satisfying $r(\a)=s(\b)$, then the \emph{composition} or the \emph{product} $\a\b$ is well-defined by concatenation
$$\a\b=a_1\cdots a_nb_1\cdots b_m.$$
The \emph{addition law of length} $|\a\b|=|\a|+|\b|$ holds obviously when $\a\b$ is well-defined.

Reversing arrows gives the \emph{dual digraph} $E^*$ of $E$.  $E^*$ has the same vertex set $E^0$ as $E$, and the set $\{a^* | a \in E^1\}$ of arrows with $s(a^*)=r(a),\, r(a^*)=s(a)$. Hence paths in $E^*$ are exactly $\a^*=a_n^*\cdots a_1^*$ for 
$\a=a_1\cdots a_n\in F(E)$. 
Therefore the set of finite paths in $E^*$ is denoted by $F^*(E)$ not by $F(E^*)$ as is expected by convention, emphasizing the fact that $^*$ is an anti-isomorphism between $E$ and $E^*$.  

\begin{definition}
\label{q1} The \emph{path algebra} or the \emph{quiver algebra} $KG$ of a digraph $G$ over a field $K$ is the $K$-vector space $KG$ with base $F(G)$ such that a product 
$\m\n\, (\m, \n \in F(G))$ is the composition of paths when $r(\m)=s(\n)$, or $0\in KG$ when $r(\m)\neq s(\n)$.
\end{definition}

\begin{definition}
The {\it extended digraph} of $E$, denoted by $\hat E$, is the digraph constructed from $E$ by adding all arrows $a^*\, (a \in E^1)$. Paths in $\hat E$ lying in $F(E)$ or in $F^*(E)$ are called \emph{real} or \emph{ghost}, respectively. 
\end{definition}

\begin{definition}
\label{q2} Let $\hat E$ be the extension of a digraph $E=(E^0, E^1, s, r: E^1\rightarrow E^0)$ by adding arrows $\overset{r(a)}{\bullet} \overset{a^*}{\longrightarrow} \overset{s(a)}{\bullet}$ for all $a\in E^1$. The \emph{Leavitt path algebra} $L_K(E)$ of a digraph $E$ over a field $K$ is the factor of $K\hat E$ by the 
so-called Cuntz-Krieger relations
\begin{enumerate}
\item \,\,\,\,  (CK1)\,\,\,\, for any two arrows $a, \, b\in E^1$\,\,\,\,\,\,\,\,\,\,\,\,\,\,\,\,\,\,\,  
 $a^*b=\begin{cases} r(a)& \text{if}\qquad a=b,\\ 0& \text{if}\qquad a\neq b\end{cases},$ 
\vskip 0.6cm
\item \,\,\,\, (CK2)\,\,\,\, for every regular vertex $v\in E^0$ \,\,\,\,\,\,\,\,\,\,\,\,\,\,\,\,\,\,\,\,\,\,\,\,\,\,\,\,\,\,\, $v=\sum\limits_{a\in s^{-1}(v)}aa^*$. 
\end{enumerate}
The \emph{Cohn path algebra} $C_K(E)$ of $E$ is the factor of $K\hat E$ by (CK1).
\end{definition}
From the definition it is clear that $KE$ and $KE^*$ are subalgebras of both $C_K(E)$ and $L_K(E)$.  The set  $\{\a\b^*\,|\, r(\a)=r(\b);\, \a,\, \b\in F(E)\}$ is a $K$-basis of $C_K(E)$ but only a set of generators for $L_K(E)$ over $K$ in view of the Cuntz-Krieger condition (CK2). However, there is, fortunately, no confusion when 
$\sum k_i\a_i\b^*_i\in K\hat E$ are used simply for elements of both $C_K(E)$ and $L_K(E)$. By Cuntz-Krieger condition (CK1) the assignment 
$\a \in F(E)\mapsto \a^*\in F^{\ast}(E)$ induces a standard, canonical involution in both $L_K(E)$ and $C_K(E)$ respectively. 

\begin{remarks} \ { } 
\label{nonunital}
\begin{enumerate}
\item Cohn and Leavitt path algebras  are in general factors
of free associative $K$-algebras without identity, that is, of the algebras of polynomials in non-commuting variables with zero constant term by ideals which are also $K$-spaces. 
\item These algebras are unital if and only if  
$E^0$ is a finite set. Moreover, for digraphs with finite vertices they are finitely generated and they are finitely presented if and only if the graphs are finite. Therefore for graphs of finite vertices Schreier techniques, or in modern terminology, Gr\"obner bases, are quite efficient tools in their study as is nicely observed in Lewin's work \cite{le1}.  
\item Finite sums of vertices act as a set of local units for both Cohn and Leavitt path algebras. 
\end{enumerate}
\end{remarks}

\section{Leavitt algebras $L_{(1,n)}$ are rings of quotients}
\label{c1}
Because of its importance, we devote this section to the particular case of Leavitt algebras, although one could incorporate this study into the next section. We refine first the nice observation by Rosenmann and Rosset \cite{rr} by showing that universal Leavitt algebras $L_K(1,n) (n>0)$ of a graph consisting of one vertex together with $n$ loops over an arbitrary field $K$ are rings of quotients of the unital free algebras of rank $n$ with respect to the Gabriel topology of certain finitely generated essential right ideals. As an application we offer an alternative elementary and direct approach to basic properties of Leavitt algebras without citing difficult deep results or quoting from related results for Cuntz algebras ${\mathcal O}_n$.

Fix an integer $n>0$ and a commutative field $K$. By definition, $L_K(1,n)$ is a Leavitt path algebra of 
a digraph with one vertex and $n$ arrows $a_i$. Consequently, $L_K(1,n)$ is the $K$ algebra generated by $a_i, a^*_i$ subject to relations $\su a_ia^*_i=1, a^*_ja_i=\begin{cases} 1& \text{if}\qquad j=i,\\ 0& \text{if}\qquad j\neq i\end{cases}$. Elements in $L_K(1,n)$ are not necessarily unique linear combinations $\sm \a_i\b^*_i$ where $\a_i$ and $\b_i$ are monomials in the $a_i$, respectively, and $*$ is the involution of $L_K(1,n)$ induced by sending $a_i$ to $a^*_i$. Moreover, the relation $a^*_ja_i=\begin{cases} 1& \text{if}\qquad j=i,\\ 0& \text{if}\qquad j\neq i\end{cases}$ implies that for monomials $\a$ and $\b$ in the $a_i$ with $|\b|\leq |\a|$ the product $\b^* \a$ is either $0$ or a monomial in the $a_i$, i.e., an element in the free algebra $A=K\langle a_1, \cdots, a_n \rangle$. Note that $L_K(1,1)$ is the Laurent polynomial algebra
$K[x, x^{-1}]$ which is obviously a ring with IBN. We next verify 

\begin{theorem} \label{flatepilea} The canonical inclusion of $A$ into $L_K(1, n)$ is a flat epimorphism, i.e., it is a flat bimorphism.
\end{theorem}
\begin{proof} In view of Theorem \ref{flatepi} it is enough to verify its assertion $(c)$. Since $A$ is a subalgebra of $L_K(1,n)$, claim $(ii)$ of $(c)$ is obvious. The case of the Laurent polynomial algebra $K[x, x^{-1}]$, i.e., the case of $L_K(1,1)$, is trivial because it is the localization of the commutative polynomial algebra $K[x]$ with respect to the set $\{1,x, x^2, \cdots \}$. However, it is worth noting that the associated (perfect) Gabriel topology is given by the filter base of ideals $x^lK[x]$ and so there are finite codimensional ideals which are not open with respect to this topology!  For the case of $n>1$ claim $(i)$ of $(c)$ is fortunately an immediate consequence of the following two observations. First, a product $\b^* \a$ is always either $0$ or a monomial in $A$ if $|\a|\geq |\b|$ where $\a$
and $\b$ are monomials in the $a_i$'s. Secondly, one has 
$$1=\su a_ia^*_i=\cdots=\sum\limits_{|\a_k|=|\b_k|=l}\a_k \b^*_k=\sum\limits_{|\a_k|=|\b_k|=l}\a_k(\su a_ia^*_i)\b^*_k=\sum\limits_{|\a_k|=|\b_k|=l+1}\a_k \b^*_k$$ 
for every $l\in \N$ where $\a_k, \b_k$ are monomials in the $a_i$'s.  
\end{proof}

From now on we assume $n>1$. By assertion $(b)$ of Theorem \ref{flatepi} a right ideal $R$ of the free associative algebra $A=K\langle a_1, \cdots, a_n \rangle$ is open in the Gabriel topology $\frak T$ determined by the canonical flat epimorphism $A\rightarrow L_K(1, n)=A_{\frak T}$ if and only if $RL_K(1,n)=L_K(1,n)$. Consequently, there are finitely many elements $r_i\in R, t_i\in L_K(1,n)$ with $\sum r_it_1=1$ whence $\sum r_iA\subseteq R$ is also open, hence $\sum r_iA$ is open in the topology $\frak T$. Consequently, $\sum r_iA$ is a finitely generated essential right ideal
of $A$ by Corollary \ref{rightid}. By (3.3) Theorem \cite{rr} $\sum r_iA$ is a right ideal of finite codimension of $A$ and so $R$ is finite codimensional, too. Again by (3.3) Theorem \cite{rr} $R$ is a finitely generated essential right ideal of $A$. This shows that a Gabriel topology $\frak T$ of $A$ defining the canonical flat epimorphism $A\rightarrow L_K(1,n)$ consists of certain finitely generated essential, i.e., finite-codimensional right ideals whence $\frak T$ is coarser than the fc-topology invented in \cite{rr} consisting of all finite codimensional right ideals of $A$. However, one has a better, visual description of perfect localizations of $A$ resulting in the Leavitt algebra $L_K(1, n)$ as follows. 
The equality $\su a_ia^*_i=1$ shows that the ideal $I=\su a_iA$ is open with respect to the topology $\frak T$ in view of $(b)$ Theorem \ref{flatepi}. Moreover $I$ is clearly a free right $A$ module of rank $n$ with canonical projections $a^*_i$. More generally,
the equalities $1=\sum\limits_{|\a_k|=|\b_k|=l}\a_k \b^*_k\,\, (l\in \N)$ where $\a_k$ and $\b_k$ are monomials in the $a_i$'s imply that ideal powers $I^l\, (l\in \N)$ are open in $\frak T$, and free right $A$-modules of rank $n^l$ with canonical projections $\b^*_k$ {\bf ($|\b_k|=l)$}. These facts show, in view of Proposition VI.6.10 \cite{s1}, that the ideal topology $\frak F$ induced by powers $I^l$ is a Gabriel topology. Moreover by definition the ring of right quotients $A_{\frak F}$ is obviously the Leavitt algebra $L_K(1,n)$. Therefore we have proved the following result.
\begin{theorem}
\label{lo1} Let $A$ be the free associative algebra $K\langle a_1, \cdots, a_n \rangle$ of rank $n>1$ over the field $K$ where $a_i$ are free variables. The topology $\frak F$ defined by powers of the ideal $\su a_iA$ is a perfect  Gabriel topology and the ring of right quotients $A_{\frak F}$ is the Leavitt algebra $L_K(1,n)$. The topology $\frak F$ is obviously the coarsest Gabriel topology defining $L_K(1, n)$.
\end{theorem}
We shall see later in Corollary \ref{topoeq} that the ideal topology $\frak F$ defined in Theorem \ref{lo1} coincides with the topology $\frak T$ induced by the canonical flat bimorphism $A\rightarrow L_K(1,n)$. Moreover,
the above argument shows also that an arbitrary basis $\{s_1, \cdots, s_n\}$ of the free right $A$-module $I_A=\su a_iA$, defines the representation of the ring $A_{\frak F}$ of right quotients as the Leavitt algebra $L_K(1,n)$ with respect to injections (partial isometries) 
$s_i: A\rightarrow s_iA: r\in A\mapsto s_1r$ and projections $s^*_i:\su s_iA\rightarrow A$ if $A=K\langle s_1, \cdots, s_n\rangle$. It is well-known and easy to verify that the $s_i$'s are algebraically independent over $K$ but unclear whether these $s_i$ generate $A$. 
In any case $L_K(1,n)=A_{\frak F}$ is isomorphic to its subalgebra generated by $s_i$ and $s^*_i\, (i=1, 2, \cdots, n)$ but this subalgebra need not coincide with $L_K(1,n)$ in general. However, if $s_i \, (i=1, \cdots, n)$ are images of the corresponding $a_i$'s under any $K$-algebra automorphism $\f$ of $A$, then the Leavitt algebra $L_K(1,n)=A_{\frak F}$ can be clearly
generated by $s_i, s^*_i\, (i=1, \cdots, n)$. This reveals a close relation between Leavitt algebras and automorphism groups of free associative algebras.
More generally, any right ideal $R$ of codimension 1 is a two-sided ideal of $A$, as is easily seen, and furthermore is a free right $A$-module of rank $n$ by the Schreier-Lewin formula \cite{le1}. Writing $a_i=r_i+k_i$ for appropriate $r_i\in R, \, k_i\in K \, (i=1, \cdots, n)$ one sees immediately that the $r_i$ generate $A$. We claim that the $r_i$ are algebraically independent over $K$. Namely, a dependence $\su r_it_i=0\, (t_i\in A)$
implies $\su a_it_i=\su k_it_i$ whence $t_i=0$ follows for all indices $i$ by the unique normal form $a=k+\su a_ib_i\, (k\in K; b_i\in A)$ for every non-zero element $a$ of $A$. This result fits with the well-known fact that only the field $K$ can be recovered from its free unital associatve algebras, but free variables are not uniquely determined!
Even the associated non-unital free associative algebras are not uniquely determined inside the unital free associative algebras. They correspond uniquely to maximal (one-sided or two-sided) one-codimensional ideals of the considered unital free associative
algebras. This result establishes a close connection between automorphisms of $A$ and ideals of codimension 1 with an appropriate predescribed basis. Therefore we obtain a coordinate-free description of Leavitt algebras together with their close relation to automorphisms of free associative algebras.
\begin{corollary}
\label{cofree} A ring $A_{\frak F}$ of right quotients of a free associative algebra $A$ of rank $n$ over a field $K$ with respect to the ideal topology defined by an ideal $I$ of codimension $1$ is isomorphic to the Leavitt algebra $L_K(1,n)$. There is a one-to-one correspondence between representations of $A_{\frak F}$ as a Leavitt algebra over $K$ generated by a suitable basis
$\{r_1, \cdots, r_n\}$ of $I$ and automorphisms of $A$.
\end{corollary} 
As applications we derive below some basic well-known properties of Leavitt algebras $L_K(1,n)$. By using Cohn's theory of free ideal rings which can be found succinctly and transparently in the last section of \cite{c3}, and repeating the insightful argument of Rosenmann and Rosset \cite{rr} on module type one, we can directly and easily re-obtain fundamental results on  $L_K(1, n)$. The argument of Rosenmann and Rosset \cite{rr} on module type provides, in particular, a new conceptual way to determine the module type of certain rings. 

\begin{theorem}\label{leav2} For $n>1$ the Leavitt algebra $L_K(1, n)$ is simple and all of its right ideals are free, whence
it is hereditary and all projective modules are free, too. Moreover, the module type of $L_K(1, n)$ is $(1,n)$ and its Grothendieck group $K_0(L_K(1,n))$ is isomorphic to the cyclic group of order $n-1$.
\end{theorem}
\begin{proof} Let $A=K\langle a_1, \cdots, a_n \rangle \, (n>1)$ be the free associative algebra over $K$ generated by free variables $a_i$ and consider $L_K(1,n)$ as the perfect localization $A_{\frak F}$ of $A$ with respect to the ideal topology $\frak F$  defined by the ideal $\su a_iA$.  If $I$ is a non-zero ideal of $L_K(1, n)$, then the ideal $I\cap A$ of $A$ is also non-zero by Corollary \ref{rightid}. Now we can use Cohn's trick presented in the proof of \cite{c1} Proposition 8.1, by considering an element $0\neq a\in I\cap A$ such that $a$, as a non-commutative polynomial in variables $a_i$, has a minimal number of nonzero coefficients and also has a minimal total degree among such elements of $I\cap A$. By multiplying on the left with certain
$a_i^*$'s and on the right with $a_i$'s if necessary, one sees clearly that such an element $a$ in $I\cap A$ must be a nonzero scalar, that is, $I=L_K(1,n)$ when $L_K(1,n)$ is simple.
 
Again by Corollary \ref{rightid} we have $J=(J\cap A)L_K(1, n)$ for every right ideal $J$ of $L_K(1,n)$. Now, by Cohn's result \cite{c3} $A$ is a free ideal ring, whence $J\cap A$ is a free $A$-module, i.e., $J\cap A=\sum b_iA=\oplus b_iA$ for appropriate elements $b_i\in A$ satisfying $b_ir=0 \Rightarrow r=0$ for each index $i$ and every element $r\in A$. Consequently, by Corollary \ref{rightid} we have $J=(J\cap A)L_K(1, n)=(\sum b_iA)L_K(1, n)=(\oplus b_iA)L_K(1,n)=\oplus b_iL_K(1,n)$ is a free right $L_K(1,n)$-module. Therefore, $L_K(1, n)$ is a hereditary algebra such that every projective module is free, i.e., a projective-free algebra.

To finish the proof we have to compute the module type of $L_K(1, n)$. We use here the streamlined argument of Rosenmann and Rosset \cite{rr} for the module type of $L_K(1, n)$ with essential simplifications. To see that $(1, n)$ is the module type of $L_K(1, n)$ it is clearly enough to show that if $L_K(1, n)$ is also free of rank $m>1$, then $n-1$ divides $m-1$. This means that there are elements $b_j\in L_K(1, n)\,\, (j=1, \cdots, m)$ satisfying $L_K(1,n)=\sum b_jL_K(1, n)=\oplus b_jL_K(1, n)$ such that all right ideals $b_jL_K(1, n)\,\, (j=1, \cdots, m)$ are isomorphic to $L_K(1, n)$, i.e., multiplication by $b_j$ on the left is injective on $L_K(1,n)$. Therefore
the right ideals $b_j{\rm dom}(b_j)\cong {\rm dom}(b_j)\, (j=1, \cdots, m)$ of $A$ are free right $A$-modules of rank $\equiv 1$ mod$(n-1)$ by the finite codimensionality of ${\rm dom}(b_j)$ and the Schreier-Lewin formula \cite{le1}. This shows that $R=\bigoplus_j b_j{\rm dom }(b_j)$ is a free $A$-module of rank 
$\equiv m$ mod$(n-1)$. On the other hand, $R$ is an essential right ideal of $A$. Namely, for any $0\neq a=\sm b_jq_j\in A\, \, (q_i\in A_{\frak F}=L_K(1, n))$ with at least one of the $q_j\neq 0$, say $q_1$, there are nonzero elements $r_1\in q_1{\rm dom}(q_1)\cap ({\rm dom}(b_1)\cap(\cap_{j=1}^m {\rm dom}(q_j))); r_2, \cdots, r_m \in A$ such that $ q_jr_1r_2\cdots r_j\in {\rm dom}(b_j)$ for all $j=2, \cdots, m$ in view of the fact that all ${\rm dom}(c_j), {\rm dom}(q_j)$ are essential right ideals of $A$ and $q_1{\rm dom}(q_1)\neq 0$ by $q_1\neq 0$. By a choice of $r_1$, one has $0\neq q_1(r_1)=q_1r_1\in {\rm dom}(b_1)$ whence $0\neq (b_1q_1(r_1))r_2\cdots r_m=b_1q_1(r_1r_2\cdots r_m)=b_1q_1r_1\cdots r_m\in b_1{\rm dom}(b_1)$ because $A$ is a domain. Similarly one gets a product such that $ar_1\cdots r_m\neq 0$ holds and all $b_jq_jr_1\cdots r_m$ are not necessarily nonzero elements of $b_j{\rm dom}(b_j)$, respectively, for all $j>1$. Therefore $R$ is both finitely generated and an essential right ideal of $A$, whence $R$ is a right ideal of $A$ of finite codimension by (3.3) Theorem \cite{rr}. Hence again by the Schreier-Lewin formula \cite{le1} $R$ is a free right $A$-module of free rank $\equiv 1$ mod$(n-1)$.
Therefore $m-1\equiv 0$ mod$(n-1)$ holds, i.e., $n$ is the smallest integer satisfying $A_{\frak F}\cong A_{\frak F}^m$, whence the module type of $L_K(1, n)=A_{\frak F}$ is $(1,n)$. Consequently, the Grothendieck group $K_0(L_K(1,n))$ is cyclic of order $n-1$ because projective modules over $L_K(1,n)$ are free, completing the proof.
\end{proof}
Theorem \ref{leav2} essentially extends the class of projective-free rings. It shows also that the determination of module type is indeed equivalent to the computation of the associated Grothendieck group for these particular rings. Typical examples for projective-free rings are local rings, principal ideal domains, free associative algebras or more generally free ideal rings and commutative polynomial algebras
over fields of finitely many variables (Serre's conjecture, now a theorem by Suslin and Quillen). 
\begin{remark} The case $n=1$ is not very interesting but worth noting because it helps us understand certain aspects of localization. The localization of $K[x]$ with respect to the Gabriel topology (consisting of all finite codimensional ideals) is the field $K[x]$ of rational functions, while $K[x, x^{-1}]$ as $L_K(1, 1)$ is the localization with respect to the Gabriel topology of all ideals containing some powers of $x$. Since the maximal ring of quotients of $K[x]$ is commutative, the Jacobson ring of one-sided inverses, i.e., the Cohn algebra of the graph of one vertex together with one loop, cannot be a (perfect) localization of its quiver algebra $K[x]$. We shall discuss this situation in the next section.
\end{remark} 
We end this section with some results on certain localizations of free associative algebras of rank $n$. Rosenmann and Rosset \cite{rr} showed that finite codimensional right ideals of a free associative algebra $A$ constitute a Gabriel topology of $A$, called an \emph{fc-topology}. By (3.3) Theorem \cite{rr}, these right ideals are exactly the finitely generated essential right ideals. Consequently, by \cite{s1} Proposition XI.3.3 and Proposition XI.3.4 (d) one obtains that the fc-localization is a perfect localization. In particular, the localization $Q^{fc}_K(n)$ is the field of rational functions in case $n=1$. If $\f: A\rightarrow B$ is an arbitrary flat bimorphism of $A$ of rank $n>1$, then any open right ideal $J$ of $A$ with respect to the topology defined by $\f$ must contain an open finitely generated essential right ideal of $A$ by Corollary \ref{rightid} and so $J$ contains an open right ideal of finite codimension by (3.3) Theorem \cite{rr}. This implies that $J$ is also finite codimensional, i.e., $J$ is open with respect to the fc-topology. This shows that $Q^{fc}_K(n)$ coincides with the maximal, even largest (with respect to inclusion) flat epimorphic right ring 
$Q_{\rm tot}(A)$ of quotients of $A$ defined in Chapter XI.4 \cite{s1}. Therefore the same argument for Leavitt algebras holds also for $Q^{fc}_K(n)$.
\begin{theorem}\label{fc} The fc-localization $Q^{fc}_K(n)$ of a free associative algebra $A$ of rank $n$ over a field $K$ is a perfect localization, even coincides with the largest flat bimorphic right ring $Q^r_{tot}(A)$ of quotients of $A$.
$Q^{fc}_K(n)$ is a simple, projective-free algebra such that its module type is either IBN for $n=1$ or $(1, n)$ if $n>1$. Furthermore, its Grothendieck group is a cyclic group of either infinite order for $n=1$ or  
$n-1$ for $n>1$.
\end{theorem}
\begin{proof} The only less trivial claim is the simplicity of $Q^{fc}_K(n)$. If $n=1$, then the quotient ring is a field of rational functions in one variable, hence it is simple. If $n>1$, then the ring of quotients with respect to fc-localization is simple by the same argument used in the first paragraph proving Theorem \ref{leav2}. Namely, any nonzero ideal $I$ of $Q^{\rm fc}_K(n)$ has a nonzero intersection with $A$ and so its elements are linear combinations of monomials in the $a_i$ and by definition the elements $a^*_i$ are also elements of $Q^{\rm fc}_K(n)$, so this intersection must contain a non-zero constant, completing the proof. 
\end{proof}
\begin{proposition}\label{max} The maximal ring of right quotients of a free associative algebra is a flat left module but
it is not a perfect localization if the rank is not 1. This means, in this case, that the canonical embedding of the free associative algebra of rank $>1$ is not a flat epimorphism. In particular, the maximal rings of right quotients of free associative algebras are simple.
\end{proposition}
\begin{proof} If the rank of the free associative algebra is not 1, then the commutator ideal is infinite codimensional. Moreover, the commutator ideal, in fact, every nonzero ideal of a free associative algebra, is always an essential hence dense ideal because free associative algebras are both domains and nonsingular rings. This shows by \cite{rr} (3.3) Theorem, and \cite{s1} Proposition XI.3.4(d) that the dense topology is not perfect. However, by
Sandomirski's result [\cite{s1} Proposition XII.6.4] the maximal ring $Q_{\rm max}(A)$ of quotients is a flat left $A$-module. For the last claim one observes that the case of rank $n=1$, that is, the case of polynomial algebra $A$ of one variable, is trivial because the maximal ring of quotients coincides with the rational function field. For the case of rank $n\geq 2$, $Q_{max}(A)$ contains infinitely many simple subalgebras isomorphic to $L_K(1,n)$, and referencing any one of them implies the claim. This completes the proof.
\end{proof}
\begin{remark}\label{fpideal}It is important to emphasize the following concerning the proof of Proposition \ref{max}. Namely, the commutator ideal of a free associative algebra of rank $n\geq 2$ is a finitely generated two-sided ideal generated by $a_ia_j-a_ja_i$ for all different pairs $(i, j)$ but not a finitely generated one-sided ideal! Therefore it is not a finite codimensional one-sided ideal, but the ordinary commutative polynomial algebras are finitely presented algebras. Moreover, finitely generated associative algebras are not necessarily noetherian algebras! 
\end{remark}
Assume now $n>1$ and consider an arbitrary flat bimorphism $\f:A\longrightarrow Q$ of a free associative $K$-algebra $A$ of rank $n$ together with the associated Gabriel topology. If 
$J$
is an arbitrary open right ideal, then it is finitely generated and essential hence finite codimensional by (3.3) Theorem \cite{rr}. Therefore the largest two-sided ideal contained in $J$ which is exactly the  annihilator ideal $I=\{a\in A\, |\, Aa\subseteq J\}$ of $A/J$, is obviously open, and finite codimensional. This implies that $A/I$ is a finite-dimensional $K$-algebra. Moreover, if $I$ is an arbitrary two-sided ideal of $A$ such that 
$A/I$ is a finite-dimensional algebra, then $I$ as a right ideal of $A$ is finitely generated by the Schreier-Lewin formula \cite{le1}. Furthermore, $I$ is also a right essential ideal because $A$ is a domain whence $I$ is a right dense and finite-codimensional ideal. Consequently, there is a bijection between flat bimorphisms 
$\f: A\longrightarrow Q$ of the free associative $K$-algebra $A$ and sets $\Phi_{\rm fs}$ of maximal two-sided finite-codimensional ideals $I\vartriangleleft A$ by \cite{s1} Proposition VI.6.10. Open right ideals of $A$ are exactly the right ideals containing some products of ideals from $\Phi_{\rm fs}$ by \cite{s1} Proposition VI.6.10. Hence the fc-topology is a Hausdorff ideal topology and $A$ is residually finite, pointing out a striking similarity with the structure of a profinite topology for groups. In particular, flat bimorphisms of free associative algebras of rank $n>1$ incorporate  both ring theory and module theory of finite dimensional algebras which are generated by at most $n$ generators. However, the completion of $A$ with respect to the fc-topology is not the direct products $\prod_I \underleftarrow{\rm lim} A/I^l$ of inverse limits $\underleftarrow{\rm lim} A/I^l$ of the canonical inverse systems $\{A/I^l\, | \, l\in \N\}$ where $I$ runs over all maximal finite-codimensional ideals of $A$, because maximal two-sided finite co-dimensional ideals are not commutable. As in the case of both $Q^{\rm fc}(n)_K$ and $L_K(1, n)\, (n\geq 2)$, if there is an open maximal one-codimensional ideal of $A$ with respect to a finest topology defined by
a flat bimorphism $A\longrightarrow Q$, then one obtains by the same proof for $Q^{\rm fc}(n)_K$, that $Q$ is a simple and right hereditary ring, projective right $Q$-modules are free, and the module type of $Q$ is $(1, n)$ whence $K_0(Q)\cong \Z/(n-1)\Z$.  In the general case of an arbitrary flat bimorphism $A\longrightarrow Q$ all that we can say is that $Q$ is right hereditary, right projective-free; i.e., right projective $Q$-modules are free. Therefore we have verified the following result.
\begin{proposition}\label{maxiflat} Let $A$ be a unital free associative $K$-algebra of rank $n\geq 1$. Then the fc-topology of $A$ is a Hausdorff topology having a basis of finite-codimensional two-sided ideals.
More generally, there is a bijection between flat bimorphisms $\f:A\longrightarrow Q$ and sets $\Phi_{\rm fs}$ of maximal two-sided finite-codimensional ideals $I$ of $A$. Namely, a two-sided ideal $I$ belongs to $\Phi_{\rm fs}$ if and only if $IQ=Q$ and $I$ is maximal. A right ideal $R$ is open in the finest Gabriel topology defined by $\f$ if and only if it contains a products of ideals from $\Phi_{\rm fs}$. $Q$ is then a right hereditary, right projective-free ring. If $\Phi_{\rm fs}$ contains a maximal one-codimensional ideal, then $Q$ is a simple ring of module type $(1,n)$.
\end{proposition}

As an immediate consequence of Proposition \ref{maxiflat} we are now in position to show that the ideal topology $\frak F$ defined by a maximal right ideal $I$ of codimension 1 and the topology $\frak T$ given by the corresponding flat bimorphism $A\rightarrow A_{\frak I}\cong L_K(1,n)$ of the free unital associative algebra $A$ of rank $n>1$ coincide.

\begin{corollary}\label{topoeq} Let $A$ be a free unital associative algebra of rank $n>1$ over a field $K$ and $\frak F$ an ideal topology defined by a maximal right ideal $I$ of codimension 1. Then $\frak F$ coincides with the toplogy $\frak T$ induced by the flat bimorphism $A\rightarrow A_{\frak I}\cong L_K(1, n)$. In particular, if $I$ is an arbitrary finite codimensional two-sided ideal of $A$, then the ideal topology $\frak F$ defined by powers of $I$ coincides with the topology $\frak T$ determined by the corresponding flat bimorphism $A\longrightarrow A_{\frak F}$.
\end{corollary}
\begin{proof} To verify the statement, it is enough to see that a right ideal $R$ of $A$ contains some power of $I$ if $RA_{\frak F}=A_{\frak F}$ holds.  Without loss of generality one can assume that $I$ is the ideal $\sum_{i=1}^n a_iA$ and so $A_{\frak T}=A_{\frak F}=L_K(1,n)$ in view of our standard notation. Assume indirectly that $R$ is not open with respect to
the topology $\frak F$, then the two-sided ideal $\ann_AA/R$ satisfies $\ann_AA/R L_K(1,n)=L_K(1, n)$ as we have already seen. Consequently, there is a maximal two-sided ideal $J\neq I$ that satisfies $JL_K(1,n)=L_K(1, n)$. Since $J$ is a free right $A$-module of finite rank and $J$ is open with respect to the topology $\frak T$, any projection $q$ of $J$ onto $A$ is also an element in $L_K(1,n)$ which is obviously not contained in $A$. Since $q$ can be expressed as a linear combination $\sum k_i \a_i\b^*_i$ where $\a_i, \b_i$ are monomials in the $a_1, \cdots, a_n$ the domain of definition of $q$ contains some power $I^l$ for some $l\in \N$. Therefore the domain of definiton of $q$ contains $J+I^l\supseteq (I+J)^l=A$, whence $q\in A$ holds which is a contradiction! Hence $\frak T=\frak F$. The last assertion can be checked in the same manner. 
\end{proof}

In view of Proposition \ref{maxiflat} we are going to determine the module type of all flat bimorphisms $\f: A\hookrightarrow Q$. By way of preparation, we consider the case when $Q$ is the ring of right quotients of $A$ with respect to the $I$-adic topology where $I$ is a maximal finite-codimensional ideal of $A$.
\begin{theorem}\label{sa} Let $I$ be a maximal ideal of  a free associative $K$-algebra $A$ of rank $n\geq 2$ over a
field $K$ such that $A/I$ is a matrix ring $D_m$ over a division ring $D$ of dimension $l$ over $K$.
If $Q$ is a ring of right quotients of $A$ with respect to the perfect Gabriel $I$-adic topology defined by the powers $I^l\, (l\in \N)$, then the module type of $Q$ is $(1, lm(n-1)+1)$.
\end{theorem}
\begin{proof}
 By the previous results of this section we already know that the canonical imbedding $A\hookrightarrow Q$ is a flat bimorphism. First we show that the $I$-adic topology coincides with the topology defined by the flat bimorphism $A\hookrightarrow Q$. For this purpose, it is clearly enough to see that a right ideal $R$ of $A$ is open in the $I$-adic topology if
$RQ=Q$ holds. Namely, the equality $RQ=Q$ implies $r_i\in R$ and $q_i\in Q$ with $\sum_i r_1q_i=1$. There is then
a positive integer $j$ that all $q_i$ are well-defined on $I^j$. Consequently, for all $a\in I^j$ we have $a=\sum_i r_i q_i(a)=\sum_i r_i(q_ia)\in R$ whence $R$ is open in the $I$-adic topology. If $R$
is the inverse image of a maximal right ideal of $A/I\cong D_n$, then $R$ has codimension $|D:K|m=lm$. Therefore $R$ has a free generator set $\{r_1, r_2, \cdots, r_{lm(n-1)+1}=r_d\}$ by the Schreier-Lewin formula \cite{le1}. For each index $i\in \{1, 2, \cdots, d\}$ let $r^*_i\in Q$ be defined on $R$ by sending $r_ia\in R$ to $a$ and all the other ${r_j}'s\, (j\neq i)$ to $0$. Then we have the identities  
$r^*_jr_i=\begin{cases} 1& \text{if}\qquad j=i,\\ 0& \text{if}\qquad j\neq i\end{cases}$ and $\sum r_ir^*_i=1$ showing an isomorphism $Q_Q\cong Q^d_Q$. Therefore
to show that $Q$ has a module type $(1, d)=(1, lm(n-1)+1)$ it is enough
to show that if $Q_Q\cong Q^m_Q \, (m>1)$, then $lm(n-1)$ divides $m-1$. To this end, we observe first the following. If $q$ is an arbitrary non-zero element of $Q$, then ${\rm dom}(q)$ is open in the $I$-adic topology, whence ${\rm dom}(q)$ contains some positive power $I^t\, (t\in \N)$. Consequently, $A/{\rm dom}(q)$ can be considered as a module over a finite-dimensional primary $K$ algebra $A/I^t$ whence the $K$-dimension of $A/{\rm dom}(q)$, i.e., the $K$-codimension of ${\rm dom}(q)$, belongs to $lm\N$. In fact, this remark holds for the codimension of any open right ideal in the $I$-adic topology.  The isomorphism $Q_Q\cong Q^m_Q$ implies the existence of $m$ elements $c_i\in Q$ with trivial right annihilator satisfying $Q=\sm c_iQ=\oplus_{i=1}^mc_iQ$. By the previous remark we have ${\rm dom}(c_i)_A\cong c_i{\rm dom}(c_i)$ having codimension $mln_i$ with appropriate positive integer $n_i$ for each index $i\in \{1, \cdots, d\}$. Therefore the right ideal $J=\sm c_i{\rm dom}(c_i)=\oplus_{i=1}^mc_iQ$ is a free right  $A$-module of rank $m+(\sm n_i)ml(n-1)$ by the Schreier-Lewin formula \cite{le1}. Consider now an arbitrary non-zero element $c\in Q$. By Theorem \ref{flatepi} (c)(1) there are elements $a_1, \cdots, a_t\in A;\ q_i\in Q$
satisfying $ca_i\in A, \sum_ia_iq_i=1$ whence $a_i\in {\rm dom}(c)$ and $c=c\sum_i a_iq_i=
\sum_i (ca_i)q_i\in (c{\rm dom}(c))Q $ hold. Consequently, one has $JQ=(\sm c_i{\rm dom (c_i)})Q= 
\sm (c_i{\rm dom}(c_i))Q=\sm c_iQ=Q$ implying that $J$ is open in the $I$-adic topology. Hence by the Schreier-Lewin formula, $J$ has the rank $tlm(n-1)+1$ for some positive integer $t\in \N$. This shows that $lm(n-1)$ divides $m-1$, completing the proof. 
\end{proof}
By considering inverse images of maximal ideals of the commutative polynomials over $K$ in $n$ variables one gets maximal two-sided finite-codimensional ideals of $A$. Since there are elements of $A$ which are not necessarily (non-commutative) polynomials in $r_i$ by comparing ranks, it seems possible that $Q$ is not simple and that there does not exist a natural, canonical involution on $Q$! It is well-known and easy to see that the $K$-subalgera of $A$ generated by the $r_i$ is a free associative algebra of rank $lm(n-1)+1$ over $K$. Furthermore it is also important to look for a normal form for elements of a flat bimorphism $A\hookrightarrow Q$ when $Q$ is a ring of right quotients of a free associative algebra of rank $n\geq 2$ over a field $K$ with respect to the $I$-adic topology and $I$ is a maximal finite-codimensional ideal of $A$ such that $A/I$ is a finite-dimensional division $K$-algebra. In this case, if $I_A=\oplus r_iA$ is a free module with respect to a basis $r_i\,(i=1,2, \cdots, d)$, then one can define the usual elements $r^*_i \in Q$ such that $r^*_jr_i=\begin{cases} 1& \text{if}\qquad j=i,\\ 0& \text{if}\qquad j\neq i\end{cases}$  and $\sum r_ir^*_i=1$ hold. It is clear that $Q$ is a $K$-algebra generated by $A$ and these $r^*_i$'s, but it is a unknown whether elements of $Q$ can be written in the form $\sum a_jb^*_j$ with suitable $a_j\in A$ and monomials $b^*_j$ among the  $r^*_i$'s.

We are now in position to determine the module type of an arbitrary flat bimorphism $A\hookrightarrow Q$ as follows.
\begin{theorem}\label{motype} Let $A$  be a free unital associative algebra of rank $n\geq 2$ over a field $K$ and $A\hookrightarrow Q$ an arbitrary flat bimorphism, i.e., there is a set $\Lambda$ of maximal two-sided finite codimensional ideals of $A$ such that $Q$ is a ring of right quotients of $A$ with respect to the Gabriel topology ${\frak T}_{\Lambda}$ induced by products of ideals from $\Lambda$.
For each ideal $I_{\l}\in \Lambda$ the factor ring $A/I_{\l}$ is a matrix ring $D_{m_{\l}}$ over a division ring $D$ of dimension $l_{\l}$ over $K$. Put $d_{\l}=l{\l}m_{\l}$ and let $d$ be the greatest common divisor of the $d_{\l}$'s. Then the  module type of $Q$ is $(1, d(n-1)+1)$. In addition, if the ideals of $\Lambda$ are commutable, then 
the completion of $A$ with respect to ${\frak T}_{\lambda}$ is the direct products $\prod_I \underleftarrow{\rm lim} A/I^l$  of inverse limits $\underleftarrow{\rm lim} A/I^l$ of the canonical inverse systems $\{A/I^l\, | \, l\in \N\}$ where $I$ runs over all elements of $\Lambda$.
\end{theorem}
\begin{proof} One can see immediately from the proof of Theorem \ref{sa} that $Q\cong Q^{d_{\l}(n-1)+1}$ as right $Q$-modules for each index $\l \in \Lambda$. Consequently, one gets $Q_Q\cong Q^{d(n-1)+1}$ by elementary number-theoretic reasoning because $d$ is the greatest common divisor of the $d_{\l}$'s. Therefore to show that $d(n-1)+1$ is the module type of $Q$ it is enough to show that $d(n-1)$ is a divisor of $m-1$ provided $Q_Q$ is isomorphic to $Q^m_Q$. To see this claim we observe the following. If $R$ is any open right ideal of $A$ with respect to the topology
${\frak T}_{\Lambda}$, then the annihilator ideal $I$ of $A/R$ is also finite-dimensional whence
$R$ is a finite-codimensional $A/I$-module, whence the $K$-codimension of $R$ is a sum of the dimensions of simple $A/I$-subfactors appearing in its composition series. Consequently, the
$K$-codimension of $R$ is a multiple of $d$. Now the same proof for Theorem \ref{sa} implies that
$d(n-1)$ is a divisor of $m-1$ when the module type of $Q$ is $(1, d(n-1)+1)$

To finish the proof we observe the following. If $I$ and $J$ are two different commutable coprime two-sided ideals, i.e., $A=I+J, IJ=JI$, of $A$ then we have
$$A=(I+J)(I+J)=I+J^2=I^2+J=I^2+J^2=\cdots=I^l+J^m \quad \forall l, m\in \N,$$ 
and 
$$I\cap J=(I\cap J)(I+J)=IJ+JI=IJ=JI.$$
By iterating these equalities one obtains the following more general equalities for finitely many pairwise coprime, commutable ideals $I_1, \cdots, I_m$
$$I^{n_i}_i +\cap_{j\neq i} I^{n_j}_j=A \quad \quad  \&  \qquad \cap I^{n_i}_i=\prod_i I^{n_i}_i \quad  \quad \forall n_i\in \N.$$
Consequently, the topology ${\frak T}_{\Lambda}$ is the same one induced by finite intersections of powers of ideals from $\Lambda$, and factors of $A$ by open ideals are finite direct sums $\prod A/I^{n_i}_i \cong A/\prod I^{n_i}_i=A/\cap I^{n_i}_i$ with $I_i \in \Lambda$. This shows a topological isomorphism $\bar Q=\prod_I \underleftarrow{\rm lim}A/I^l$ where $\bar Q$ is the completion of $Q$ with respect to ${\frak T}_{\Lambda}$, because they are inverse limits of the same inverse system! This completes the proof.
\end{proof}
It is worthwhile to note that it is not a routine exercise to precisely write down the $2(d(n-1)+1)$ elements $x_i,\,  y_i(=x^*_i) \, (i=1, \cdots, d(n-1)+1)$ of $Q$ in Theorem \ref{motype} satisfying the equalities $y_jx_i=\begin{cases} 1& \text{if}\qquad j=i,\\ 0& \text{if}\qquad j\neq i\end{cases}$ and $\sum x_iy_i=1$
which imply that $Q$ has the module type $(1, d(n-1)+1)$. It is also unclear whether the $x_i$'s can be taken from $A$.

Since one can embed free associative algebras in division rings, there are localizations of free associative algebras in Cohn's sense (see also Lambek \cite{lab} and Morita \cite{mori4}) which are ring epimorphisms. These localizations are, however, no longer localizations in Gabriel's sense.

\section{Leavitt path algebras of finite digraphs are flat bimorphisms}
\label{fdigraph}   

We primarily consider in this section Leavitt path algebras of finite digraphs. The main aim of this section is to prove   
\begin{theorem}
\label{geflatepi} Let $E$ be a finite digraph whence every vertex is either regular or a sink. Then the inclusion of the quiver algebra $KE$ in the Leavitt path algebra $L_K(E)$ is a flat epimorphism. Consequently, $L_K(E)$ is the localization of $KE$ with respect to the Gabriel topology of all right ideals of $KE$ which generate the right ideal $L_K(E)_{L_K(E)}$.  
\end{theorem}
\begin{proof} Since $KE$ is a subalgebra of $L_K(E)$, we have only to verify condition $(i)$ of assertion $(c)$ in Theorem
\ref{flatepi}. Although the idea of the proof is the same as for Theorem \ref{flatepilea} we have to refine the argument used there. Consider an arbitrary nonzero element $r\in L_K(E)$. Write $r=\su k_i\a_i\b^*_i (0\neq k_i \in K)$ as a linear combination with possibly the smallest number of $0\neq k_i\in K$, and among them with possibly smallest $\max \{|\b_i|\}$ where $\a_i, \b_i$ are paths in $F(E)$ with $r(\a_i)=r(\b_i)$ for each $i$, and all $\b_i$ have positive length. It is obvious that every vertex $v$ which is not a source of any
$\b_i$ for all indices $i$, satisfies either $rv=0$ or $rv\in KE$. In particular, one has $rv, r\m \in KE$ for each sink $v$ in $E$ and every path $\m={\m}v \in F(E)v$ ending in $v$. Consequently, to verify condition $(i)$ of $(c)$ in Theorem \ref{flatepi}, it is enough to see that if $v$ is a source of any one of
the paths $\b_i$, then there are paths $\m_j, \n_j \in F(E); r(\m_j)=r(\n_j); rm_j\in KE$ such that  $v=\sum \m_j \n^*_j$ holds. 
Then $v$ is a regular vertex, whence one can write
$v=\sum \m_i\n^*_i$ in view of Cuntz-Krieger condition (CK2) for appropriate $\m_i, \n_i\in F(E)$ satisfying $v=s(\m_i)=s(\n_i)$ and $r(\m_i)=r(\n_i)$. If  $r\m_i \notin KE$, then $u_i=r(\m_i)=r(\n_i)$ is not a sink by the previous remark. Therefore $u_i$ is a regular vertex and one can substitute $\m_i\n^*_i=\m_iu_i\n^*_i=\m_i\sum\limits_{s(a_j)=u_i}a_ja^*_j\n^*_i=\sum\limits_{s(a_j)=u_i} (\m_ia_j)(a^*_j\n^*_i)=\m_i\n^*_i$ for $\m_i\n*_i$ in the sum representation of $v$ and continue our process. After finitely many steps one obtains that every path $\m_i$ either ends in a sink, whence $r\m_i\in KE$, or has a length so big such that $r\m_i$ is also a linear combination of real paths, i.e., $r\m_i \in KE$. This completes the verification of condition (i) of assertion (c) in Theorem \ref{flatepilea}, whence the proof is complete.
\end{proof}

\begin{remarks}
\label{ex}
\begin{enumerate}
\item Note the important fact that lengths of paths ending at the same sink can go to infinity; that is, they can be arbitrarily long. On the other hand, there could exist sinks to which any path from $v$ does not contain a closed subpath. To demonstrate the effect of the algorithm, writing $v=\sum \m_i \n^*_i$ as a linear combination with $r\m_i\in KE$ one can assign to $v$ a set $E^v_{ex}$ of sinks $u$ such that every path from $v$ to $u$ does not contain closed paths, and let $N(v)$ be the maximum length of all paths from $v$ to sinks in $E^v_{ex}$. Let $N(v)=0$ when $E^v_{ex}$ is the empty set. These invariants may be of use for further study of Leavitt path algebras.
\item For the visualization of the argument presented in the proof of the last theorem one can consider the case of the digraph

$
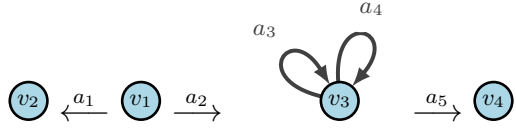
{\begin{tikzpicture}
	\Vertex[size = 0.5, label = $v_2$]{v_2}\end{tikzpicture}}\overset{a_1}{\longleftarrow} {\begin{tikzpicture}
	\Vertex[size = 0.5, label = $v_1$]{v_0}\end{tikzpicture}} \overset{a_2}{\longrightarrow}{\begin{tikzpicture}
	\Vertex[size = 0.5, label = $v_3$]{v_3}
	\Edge[label=$a_4$, position = above, loopshape = 45, loopposition = 72, Direct = true](v_3)(v_3)
	\Edge[label=$a_3$, position = {above left=2mm}, loopshape = 45, loopposition = 144, Direct = true](v_3)(v_3)
	%\Edge[style = dotted, loopshape = 45, loopposition = 144 + 72, Direct = true](A)(A)
\end{tikzpicture}} \overset{a_5}{\longrightarrow} {\begin{tikzpicture}
	\Vertex[size = 0.5, label = $v_4$]{v_4}\end{tikzpicture}}$

and $r=a_2a^*_3a^*_4+a^*_3a^*_2$ and write down the associated expression for $v_1, v_2$. Note the equality $E^{v_1}_{ex}=\{v_2\}$ and there are paths of arbitrary length from $v_1$ to the sink $v_4$.
\item If $E$ is the Dynkin graph $\overset{v_0}{\bullet} \overset{a_1}{\longrightarrow} \overset{v_1}{\bullet}\longrightarrow \cdots \overset{v_{n-1}}{\bullet} \overset{a_n}{\longrightarrow} \overset{v_n}{\bullet}$, then $KE$ is the ring of $n\times n$ upper triangular matrices and $L_K(E)$ is the matrix ring $K_n$ of $n\times n$ matrices over $K$. In this case $L_K(E)$ is the right maximal ring of quotients of $KE$. Since $KE$ is finite dimensional, $0$ is also finite codimensional and hence the ring of quotients of $KE$ with respect to the Gabriel topology of $KE$ is trivial.
\item If $E$ is the digraph
$\bullet\,{\overset{a}{\longleftarrow}}\,\overset{v}{\bullet}\, {\overset{b}{\longrightarrow}}\, \bullet$, then its Leavitt
path algebra is isomorphic to $K_2\oplus K_2$ which is again the right maximal ring of quotients of $KE$.
\item The proof of Theorem \ref{geflatepi} also shows the importance of the Cuntz-Krieger condition (CK2). One can endow
$L_K(E)$ with the natural $\Z$-graded structure. All vertices and terms $\a \b^*$ with $\a, \b \in F(E), r(\a)=r(\b)$ and $|\a|=|\b|$ are of degree $0$. However, without (CK2), one cannot represent regular vertices as linear combinations of proper
terms $\a \b$. Therefore, the inclusion of the quiver algebra into the associated Leavitt algebra is neither forced to be an epimorphism nor flat even for an infinite digraph with a finite set of vertices. 
\end{enumerate}
\end{remarks}
As the first important consequence we present a description of Leavitt path algebras similar to both Cuntz' \cite{cu} construction of $\mathcal O_n (n>1)$ and Raeburn's \cite{rae} definition of graph operator algebras which captures the naturality of Cuntz-Krieger conditions (CK1) and (CK2). One also sees the utility of our different construction.
For instance
in $C^*$-algebras, as adjoints of operators on Hilbert spaces, $a^* (a\in E^1)$ are defined globally, while in our algebraic setting, $a^* (a\in E^1)$ are defined partially, whence one needs to  identify them for a ring structure.   Moreover, at the same time one can ask for an intrinsic determination of open right ideals of the quiver algebra $KE$ given by assertion $(b)$ of Theorem \ref{flatepi}. Fortunately, it turns out that the answers for both tasks are the same. We have seen from examples that the associated Gabriel topology admits a basis consisting of certain finitely generated essential right ideals containing all sinks. It is fortunate that we don't need to invoke the Gabriel topology to compute the Leavitt path algebra as the ring of (right) quotients of the ordinary quiver algebra. All we need is the easy observation 
that the right ideal $I$ of $KE$ generated by all arrows together with the sinks, is an open right ideal in the Gabriel topology given in $(b)$ of Theorem \ref{flatepi} whence $I$ is a dense right ideal of $KE$. In particular, $I$ is even a two-sided ideal of $KE$. Note that all these results are already checked as a consequence of Theorem \ref{flatepi} and Theorem \ref{geflatepi}, but one can see it directly and easily even in the case of an arbitrary digraph. Consequently, 
$KE$ can be embedded in $\Hom_{KE}(I, KE)$ by multiplication on the left with elements of $KE$ because the left annihilator
of $J$ is trivial, i.e., $0$. In addition, as a left module over $KE$, $\Hom_{KE}(I, KE)$ is generated by $KE$ and homomorphisms (functions) $a^*:J\rightarrow KE$ for $a\in E^1$ induced by sending each path $\g \in F(E)\cap J $ to either $0$ if $\g$ does not start with $a$ or to $\l \in F(E)$ if $\g=a\l$. It is worth noting that all arrows are contained in $I$ whence for an arbitrary arrow $a$ with $u=r(a)$ the vertex $u=a^*(a)=a^*a$ belongs to values of the function $\a^*:J\rightarrow KE$, but $u=r(a)$ belongs to $J$ only in the case when $u=r(a)$ is a sink. Consequently, in case of regular vertex $u=r(a)$,  $a^*$ is not defined on $u$ although the functions $u, a^*a$ are equal not only on $uI$ but on $uKE$, too. 
Therefore these homomorphisms $a^* (a\in E^1)$ trivially satisfy Cuntz-Krieger condition (CK1) for all vertices and (CK2) for regular vertices when considering them as functions from $I$ to $KE$. If a vertex $v$ is a sink, then condition (CK2) becomes empty. Note the fact that for an infinite emitter $v$ there are infinitely many homomorphisms $a^* (a\in s^{-1}(v))$, and so the canonical projections $aa^*$ provide an infinite direct sum decomposition for the vector space $vJ$ showing that there is room for using topology and restriction.

Domains of definition of partial linear transformations $\sum k_i\m_i\n^*_i (0\neq k_i \in K; \m_i, \n_i \in F(E), r(\m_i)=r(\n_i))$ from subspaces of $KE$ into $KE$, exactly form the set of open right ideals for our Gabriel topology in the case of finite digraphs. In this case, it is clear that powers $I^l$ are open and domains of definition of $\sum k_i\m_i\n^*_i (0\neq k_i \in K; \m_i, \n_i \in F(E), r(\m_i)=r(\n_i))$ contain almost all powers $I^l$. This implies, in view of Proposition VI.6.10 \cite{s1}, that for finite digraphs $E$, powers of the ideal $I$ generated by all arrows and sinks form a coarsest Gabriel topology whose localization is the Leavitt path algebra $L_K(E)$. Unfortunately, for infinite digraphs, i.e., digraphs with either infinitely many vertices or arrows, the $I$-adic topology defined above is no longer a Gabriel topology. However, as we already mentioned, domains of definition of partial linear transformations $\sum k_i\m_i\n^*_i (0\neq k_i \in K; \m_i, \n_i \in F(E), r(\m_i)=r(\n_i))$ contain almost all powers $I^l$ which are all dense ideals of $KE$. Therefore for an infinite digraph $E$ the Leavitt path algebra $L_K(E)$ is exactly the subring of the maximal right quotient ring $Q^{\rm{r}}_{\rm{max}}(KE)$ generated by $KE$ and partial functions $a^*\,(a\in E^1)$ defined above. For finite digraphs, the Gabriel $I$-adic topology is Hausdorff if and only if there are no sinks. In summary, for finite digraphs one can realize $L_K(E)$ as a perfect localization of $KE$ revealing a close connection between their module categories. For infinite digraphs with infinite sets of either vertices or arrows, the situation is still a mystery. All we can say for certain is that $L_K(E)$ is a ring of quotients of $KE$ in Utumi's sense \cite{u1}. We will return in detail to the case of arbitrary digraphs in Proposition \ref{wflatepi} and Corollaries \ref{utumi}, \ref{nonsingular}. We summarize the discussion above in

\begin{corollary}\label{idgab} Let $E$ be an arbitrary finite digraph and $I$ the two-sided ideal of the ordinary quiver algebra $KE$ over a field $E$ generated by arrows and sinks. Then powers $I^n \,(n \in \N)$ of $I$ define a coarsest perfect Gabriel topology of $KE$ whose localization is the Leavitt path algebra $L_K(E)$. The $I$-adic topology is clearly Hausdorff if and only if $E$ has no sinks.
\end{corollary}

We now directly derive a well-known result that Leavitt path algebras of finite digraphs are hereditary.
\begin{proposition}\label{hered} A Leavitt path algebra $L_K(E)$ of a finite graph $E$ over a field $E$ is hereditary.
\end{proposition}
\begin{proof} It is enough to see that every right ideal of $L_K(E)$ is projective. Namely, if $J$ is an arbitrary right ideal of $L_K(E)$, then $J=(J\cap KE)L_K(E)$ holds in view of Theorem \ref{rightid}. Since $_{KE}L_K(E)$ is flat, one has $J=(J\cap KE)L_K(E)=[J\cap L_K(E)]\otimes_{KE}L_K(E)$ whence the heredity of $L_K(E)$ follows from the fact that quiver algebras of finite digraphs are hereditary. In particular, a starting point in the theory is that every right module $M$ over $KE$ admits the following standard projective resolution

$$0\rightarrow \bigoplus\limits_{a\in E^1} Ms^{-1}(a)\otimes_Kr^{-1}(a)KE\overset{f}{\longrightarrow} \bigoplus\limits_{v\in E^0} Mv\otimes_K vKE\overset{g}{\longrightarrow} 0$$
where $f(x\otimes r)=x\otimes ar - xa\otimes x$ for each $a\in E^1, x\in Ms^{-1}a, r\in r^{-1}(a)KE$ and $g(xv\otimes vr)=xvr$ for all $v\in E^0, x\in M, r\in KE$. 
This implies that every module over $KE$ has projective dimension at most 1, that is, $KE$ is hereditary, completing the proof. 
\end{proof}
It is worth pointing out that the standard projective resolution described above remains true for a unitary module over an arbitrary digraph. However, we do not know if this result implies the heredity of a corresponding Leavitt path algebra, even for the Leavitt path algebra  $L_K(1, \infty)$  of one vertex $v=1$ and infinitely many loops $a_i\, (i\in \N)$. Namely, if $A$ is a subalgebra generated by $a_i$, then $(1-a_1a^*_1)L_K(1, \infty) \neq \{[(1-a_1a^*_1)L_K(1, \infty)] \cap A\}L_K(1, \infty)$ by $\{(1-a_1a^*_1)L_K(1, \infty)\} \cap A=\sum_{i\geq 2}a_iA$ and 
$1-a_1a^*_i\notin \sum_{i\geq 2}a_iL(1, \infty)$, as is routine to check. Another way to verify the heredity property of $KE$ for an arbitrary digraph $E$ is provided by
Schreier's technique as is presented by Lewin \cite{le1} and Rosenmann and Rosset \cite{rr}. In view of its simplicity and influential role we recall here a detailed construction.

Let $E$ be an arbitrary digraph and $A$ be its ordinary quiver algebra over a field $K$. Then the set $F(E)=\cup_{n\geq 0} F_n(E)$ of paths in $E$ where $F_n(E)\, (E^0=F_0(E), E^1=F_1(E))$
is a set of paths of length $0\leq n\in \N$. 

Schreier bases for free associative algebras introduced by Rosenmann and Rosset \cite{rr} can be extended naturally to quiver algebras of digraphs as follows.

\begin{definition}\label{schreier} Let $A=KE$ be an ordinary quiver algebra of a digraph $E$ over  a field $K$. A \emph{Schreier basis} for an arbitrary right ideal $R$ of $KE$ is a subset  $B=B_R\subseteq F(E)$ 
 that spans a right vector space $V=V_R$ that is complementary to $R$ (that is, $A=V+R, V\cap R=0$), and which is closed to taking heads. For each $n\in \N$, let $V_n$ be a right vector space generated by elements of $B$ having length at most $n$. A Schreier basis $B$ is called a \emph{strong Schreier basis} for $R$ if every path in $F_n(E)$ of length $n$ lies in $V_n+R$. 
\end{definition}

The argument of Rosenmann and Rosset \cite{rr} $(3.2)$ Lemma is used to show the existence of Schreier bases.
\begin{proposition}
\label{schreier1} There exists a strong Schreier basis $B_R$ for any right ideal $R$ of $A=KE$.
\end{proposition}
\begin{proof} The case $R=A$ is trivial because $B_A$ is just the empty set. Therefore one can assume without loss of generality that $R$ is a proper right ideal, i.e.,
 $E^0=F_0(E) \not \subseteq R$. A Schreier basis $B=B_R$ is constructed inductively by first taking a maximal $K$-linearly independent subset, $_0B$, of $E^0=F_0(E)$ modulo $R$ and letting $V_0$ be a $K$-subspace of $A$ spanned by $_0B$. Therefore $V_0+R= KE^0+R= KF_0(E)+R$ holds. If $V_0+R=A$, let ${_1B} = {_0B}$. If $V_0+R\neq A$, then $KF_1(E)$ is not a subspace of $V_0+R$. Namely, $KF_1(E)\subseteq V_0+R=KE^0+R$ would imply $KF_2(E)\subseteq KE^0F_1(E)+RF_1(E)\subseteq KF_1(E)+R\subseteq V_0+R$ and so for all $n$, $KF_n(E)\subseteq V_0+R$ holds, hence $A=V_0+R$, a contradiction. Consequently, if $V_0+R\neq A$, then let $_1B^{\prime}$ be a maximal subset 
$$\{va=a\, |\, v\in _0B \quad a\in E^1=F_1(E) \, \, \& \,\, s(a)=v\}$$
of $F_1(E)$ such that $_1B^{\prime}$ is linearly independent over $K$ modulo $V_0+R$. Put 
${_1B} = {_0B}\, \cup\,\, {_1B}^{\prime}$ and let $V_1$ be a $K$-space spanned by $_1B$. Then $KE^0+KE^1=KF_0(E)+KF_1(E)\subseteq V_1+R$ holds. Therefore one has the equality $KF_0(E)+KF_1(E)+R=V_1+R$. Assume now that $_nB^{\prime}, \, _nB$ and $V_n\, (n>0)$ have been already constructed such that $KF_0(E)+KF_1(E)+\cdots + KF_n(E) + R = V_n+R$. 

If $V_n+R=A$, let $_{n+1}B={_nB}$. If $V_n+R\neq A$, then $KF_{n+1}(E)$ is not a subset of $V_n+R$ because, as above, one can verify easily that $KF_{n+1}(E)\subseteq V_n+R$ would imply $V_n+R=A$, a contradiction. Consequently, in case $V_n+R\neq A$, let  $_{n+1}B^{\prime}$ be a maximal subset
$$\{ba \, | \,\, a\in E^1=F_1(E),\,\, b\in { _nB}^{\prime}\}$$
which  is a linearly independent set over $K$ modulo $V_n+R$. Then define $_{n+1}B={_nB}\, \cup \,\, _{n+1}B^{\prime}$ and let $V_{n+1}$ be a $K$-space spanned by $_{n+1}B$. Hence $KF_{n+1}(E)\subseteq V_{n+1}+R$ holds. If $_{n+1}B^{\prime}=\emptyset$, the process stops at this step and we define $B_R={_nB}$. If the process does not stop after finitely many steps define
$$ B = B_R = \bigcup_{n=0}^{\infty} {_nB} \,.$$
$B$ is clearly a strong Schreier basis of $R$, completing the proof.  
\end{proof}
Schreier's technique \cite{le1} is now suitable to show that right ideals of $KE$ are direct sums of cyclic right ideals isomorphic to cyclic right ideals $vKE$'s generated by vertices $v\in E^0$.

Let $\p:A=V\oplus R\rightarrow V$ be the canonical projection of $A$ along $R$ onto the  $K$-space $V$ spanned by a Schreier basis $B$ of $R$ constructed above. Then for every element $x\in A$ one has $x-\p(x)\in R$, whence $xy-\p(x)y\in R$ holds for arbitrary elements $x, y\in A$. In particular, the equality
\begin{align}\label{congthuc1}
\p(xy)=\p(\p(x)y) \quad \forall \quad x, \, y, \,\in A 
\end{align}
holds. Consequently, for every path $\m \in B$ and $a\in E^1$ the element $\m a$ is either contained in $B$ whence 
$\p(\m a)=\m a$ and so $\m a-\p(\m a)=0$, or not contained in $B$. In that case, by the construction of $B$, or equivalently, by the definition of a strong Schreier basis, $0\neq \m a-\p(\m a)\in R$ holds. Therefore, for every path $\m$ of length $l\geq 0$ in $B$ the associated element 
\begin{align}\label{cong2}
u_{\m, a}=\m a-\p(\m a)\in R \quad \forall \, \m\in B\,\, \& \,\, a\in E^1
\end{align}
is either 0 or a nonzero element of $R$. We are now ready to give another proof for the fact that quiver algebras are hereditary showing the beauty of Schreier-Lewin techniques for digraphs.
\begin{theorem}\label{free1} If $R$ is a right ideal in a quiver algebra $A=KE$ of an arbitrary digraph $E$ over a field $K$, then $R$ is isomorphic to a direct sum of right ideals generated by vertices.
\end{theorem}
\begin{proof} Since the case $R=A$ or $R=0$ is obvious, one can assume without loss of generality that $R$ is a proper nonzero right ideal. Let $B=B_R\subseteq F(E)$ be a strong Schreier basis of $R$. We show first that the nonzero elements
$u_{\m, a}=\m a-\p(\m a)\, (\m \in B,\, a\in E^1)$ generate $R$. For an arbitrary path $\b$ and $a\in E^1$, we have by \eqref{congthuc1}
$$\p(\b)a-\p(\b a)=\p(\b)a-\p(\p(\b)a).$$
Writing $\p(\b)=\smj k_j\g_j$ for some  $\g_j\in B_R$ and $k_j\in K$ one has by $\p(\b)a=\smj k_j\g_j$
\begin{align}\label{cong3}
\p(\b)a-\p(\b a)=\p(\b)a-\p(\p(\m)a)=
\smj k_j(\g_ja-\p(\g_ja))=\smj k_ju_{\g_j, a}.
\end{align}
The canonical projection $1-\p$ of $A$ onto $R$ via the decomposition $A=V\oplus R$ sends every
path $\b=c_1\cdots c_n \in F(E)\, (c_i\in E^1)$ to $(1-\p)(\b)=\b-\p(\b)\in R$, which belongs also to $\sum_{\m \in B\, \&\,  a\in E^1} u_{\m, a}A$ in view of \eqref{cong3} and the following formula
\begin{align}\label{cong4}
(1-\p)(\b)=\b-\p(\b)=\suo \{ \p(h_{\b}(j))c_{j+1}-\p(h_{\b}(j+1))\} t_{\b}(j+1).
\end{align}
Consequently, by the linearity of $\p$ together with $\b-\p(\b)$, the image $x-\p(x)$ of an arbitrary element $x\in A$ is contained in 
$\sum_{\m \in B\, \&\,  a\in E^1} u_{\m, a}A$. Hence $u_{\m, a}\, (\m \in B;\,\, a\in E^1)$ generate $R$.

Therefore to complete the proof, it remains to show that the sum $\sum_{\m \in B\, \&\,  a\in E^1} u_{\m, a}A=R$ is direct and each direct summand $u_{\m, a}A$ is isomorphic to a right ideal $r(a)A$. For the latter claim we observe that for every path $\m \in B$ and vertex $v\in E^0$ a product $\m v$ is either $\m \in B$ if $v=r(\m)$ or 0 otherwise. Hence in view of \eqref{cong2} for every vertex $v\neq r(a)$ one has 
$$ \m av=0\Longrightarrow u_{\m, a}v=-\p(\m a)v\in R\cap V=0\Longrightarrow u_{\m, a}v=0=\p(\m a)v \quad \forall \, v\neq r(a),$$
whence $u_{\m, a}A$ is isomorphic to $r(a)A$. For the first claim that $R$ is the direct sum of submodules $u_{\m, a}A\, (\m \in B=B_R;\,  a\in E^1)$ we use the nice argument of [\cite{rr}, p. 363] instead of repeating the lengthy opaque argument of Lewin \cite{le1}. Consider finitely many nonzero elements $u_{\n_j, a_i}$ defined by $\n_j\in B_R$ and $a_i\in E^1$ 
and assume a linear dependence relation $\sum u_{\n_j, a_i}x_{ij}=0$ for appropriate
elements $x_{ij}\in A$. Therefore these paths $\n_ja_i$ do not belong to $B$. By the above argument one can assume without loss of generality that each $x_{ij}$ belongs to $r(a_i)A$, respectively. We are going to show that under this extra assumption the $x_{ij}$'s are all zero. Namely, if $k\g \,\, (0\neq k\in K)$ is a path of longest length among paths represented as linear combinations of the $x_{ij}$'s starting from $s(a_i)$, say, it is a monomial of $x_{\underline{i}\, \underline{ j}}$, then the term $k\n_{\underline{j}}a_{\underline{i}}\g$ cannot cancel. This is an immediate result of the following observations. Because of the maximal length of $\g$, $\n_{\underline{j}}a_{\underline{i}}$ is not a head of any of the paths $\n_ja_i$ with $(\n_j, a_i)\neq (\n_{\underline{j}} a_{\underline{i}})$. And not being an element of $B=B_R$, $\n_{\underline{j}}a_{\underline{i}}$ is also not a head of any path in $\p(\n_ja_i)$, which would place it in $B$ (without any exception, even including the case  $(j, i)=(\underline{j}, \underline{i}))$. This contradiction finishes the proof.
\end{proof}

Theorem \ref{geflatepi} together with Proposition \ref{hered} imply also a short, elementary and direct proof for the description of the Grothendieck group $K_0(L_K(E))$ of finite digraph $E$ given in \cite{aas} Theorem 6.1.9. For another
application, recall that the Cohn path algebra $C_K(E)$ of a digraph $E$ is the factor of the path algebra $K{\hat E}$ subject to the Cuntz-Krieger relation (CK1). Since every Cohn path algebra $C_K(E)$ of a finite digraph $E$ can be realized as a Leavitt path algebra of another finite digraph (see \cite{aas} Definition 1.5.16 and Theorem 1.5.18), one has
\begin{corollary}\label{cohn} A Cohn path algebra of a finite digraph can be obtained as a ring of quotients with respect to a perfect localization and hence it is also hereditary.
\end{corollary}
Now the Jacobson algebra of one-sided inverses can be considered as a Cohn path algebra of the graph with one vertex and one loop. As an extension of a polynomial algebra $K[x]$, it is not a flat epimorphism but we can view the Jacobson algebra
of one-sided inverses as a Toeplitz algebra, i.e., the Leavitt path algebra of the following so-called Toeplitz graph $J$

%{\begin{tikzpicture}
	%\Vertex[size = 0.4, label = v]{v}\end{tikzpicture}} \overset{b}{\longleftarrow}
$
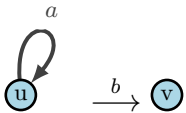
{\begin{tikzpicture}
	\Vertex[size = 0.4, label = u]{u}
	\Edge[label=$a$, position = above, loopshape = 45, loopposition = 72, Direct = true](u)(u)
           \end{tikzpicture}}\overset{b}{\longrightarrow}
{\begin{tikzpicture}
	\Vertex[size = 0.4, label = v]{v} \end{tikzpicture}}$

	%\Edge[label=$a_4$, position = {above left=2mm}, loopshape = 45, loopposition = 144, Direct = true](v_3)(v_3)
	%\Edge[style = dotted, loopshape = 45, loopposition = 144 + 72, Direct = true](A)(A)
%\end{tikzpicture}}$

Therefore the Toeplitz algebra is a perfect localization of $KJ$ which can be realized as the upper triangular matrix ring
$
 R=\begin{pmatrix}K[x]&K[x]\\
0&K
\end{pmatrix}$. Consequently, $R$ has several remarkable properties, and from these properties one can deduce several nice properties for the Toeplitz algebra by localization, offering another route to Gerritzen's results \cite{g1}.

We shall apply results of this work to module theory over Leavitt path algebras by using the advanced theory of modules over quiver algebras in subsequent papers. 
We complete this section with some technical but useful properties of Leavitt path algebras of a not necessarily finite
digraph. First we present a weak form of $(i)$ in assertion $(c)$ of Theorem \ref{flatepi}.

\begin{proposition}\label{wflatepi} For every nonzero element $r\in L_K(E)$ where $E$ is an arbitrary digraph, there is an
element $a\in KE$ that satisfies $0\neq ra\in KE$.
\end{proposition}
\begin{proof} We use induction on the minimal number $l$ when $r$ is represented as a linear combination $r=\sum\limits_{i=1}^{l} k_i\a_i\b^*_i \,\,(0\neq k_i\in K; \a_i, \b_i\in F(E), r(\a_i)=r(\b_i))$. The claim holds obviously
in the case $l=1$. Assume now the claim for all integers smaller than $l$. Without loss of generality one can assume $rv=r$ for some appropriate vertex $v\in E^0$ and there is a path $\b_i$ of positive length, otherwise $r\in KE$ whence the claim follows. The minimality of $l$ shows that the paths $\a_i$ with $|\b_i|=0$ are linearly independent, i.e., they are pairwise different. Moreover if there is an arrow $a\in s^{-1}(v)\setminus \{h(\b_i)\}$, then $0\neq ra\in KE$ holds. Consequently, one can assume the equality $s^{-1}(v)=\{h(\b_i\}$. If $|s^{-1}|>1$, one can write
$$r=x+\sum\limits_{c_j\in s^{-1}(v)}z_jc^*_j \quad (z_j\in L_K(E)).$$
Since each $rc_j$ has fewer terms than $r$, by the induction hypothesis, one need only consider the case when all $rc_j=0$, whence
$$r=rv=r(\sum\limits_{c_j\in s^{-1}(v)}c_jc^*_j)=0,$$
a contradiction. Therefore it remains to consider the case when w$s^{-1}(v)=\{c\in E^1\}$ is a one element set. In this case consider the element $$rc=\sum\limits_{i=1}^{l} k_i\a_i\b^*_ic=\sum\limits_{i=1}^{l} k_i\a_i{\bar \b}^*_i \,\,(0\neq k_i\in K; \a_i, \b_i\in F(E), r(\a_i)=r(\b_i))$$  $$=\max\{|{\bar \b}_i|\}=\max\{|\b_i|\}-1.$$ The equality $v=cc^*=c^*c$ clearly implies that $r_1=rc\neq 0$. Therefore, repeating the above argument for $r_1$ (that is, use induction on $\max\{|\b_i|\}$) one can see after finitely many steps that the claim for $r$ holds and our proof is complete.
\end{proof} 
\begin{corollary}\label{utumi} The (not necessarily unital) Leavitt path algebra $L_K(E)$ of an arbitrary digraph $E$ over a field $K$ is a ring of right quotients of the (not necessarily unital) quiver algebra $KE$ in Utumi's sense.
\end{corollary}
\begin{proof} Consider two elements $0\neq q_1, q_2$ of $L_K(E)$. By Proposition \ref{wflatepi} there is a path $\a$ in $E$ such that $0\neq q_1\a \in KE$. If $q_2\a$ belongs to $KE$, then we are all set. If $q_2\a \notin KE$, then the range of $\a$
is not a sink, because for any path $\g$ ending at a sink in $KE$ one has $q\g \in KE$ for every element $q\in L_K(E)$. Hence there is an arrow $a_1$ from the range of $\a$ with $0\neq q_1\a a_1$ and one can continue our process. After finitely many steps, the process must stop, in the sense that there must be a path $\b$ in $E$ such that $0\neq q_1\b$ and $q_2\b\in KE$, completing the proof.
\end{proof} 
As a consequence, one can reobtain \cite{aas} Proposition 2.3.7 directly in the following
\begin{corollary} \label{nonsingular} The Leavitt path algebra $L_K(E)$ of an arbitrary digraph is right non-singular.
\end{corollary}
\begin{proof} By Proposition \ref{wflatepi}, for every essential right ideal $J$ of $L_K(E)$ the right ideal $J\cap KE$ of $KE$ is essential. Therefore $L_K(E)$ is nonsingular by way of the non-singularity of $KE$ which is an obvious consequence of the heredity of $KE$ as we already remarked after the proof of Proposition \ref{hered}. This completes the proof.
\end{proof}
Since it is well-known and, in fact, easy to describe digraphs whose quiver algebras are either noetherian or artinian, by \cite{s1} Proposition XI.3.9 we have immediately
\begin{corollary}\label{chaincon} If $E$ is a finite digraph, then the Leavitt path algebra $L_K(E)$ is right noetherian or right artinian if and only if the quiver algebra $KE$ has the given property. Consequently, if a Leavitt path algebra $L_K(E)$ of a finite digraph $E$ is artinian, then $L_K(E)$ is semisimple, and as such is a right maximal ring of quotients of $KE$.
\end{corollary}
Having in mind the trivial fact that $L_K(E)$ is neither right noetherian nor right artinian if $E$ contains an infinite emitter, by using the usual techniques of representing Leavitt path algebras as direct limits of Leavitt path algebras of finite subdigraphs, we can again deduce characterizations of digraphs whose Leavitt path algebras satisfy some form of chain condition as they are presented in \cite{aas} Section 4.2 on pages 158 -- 167.
\bibliographystyle{amsplain}

\end{document}